\documentclass[12pt]{amsart}
\usepackage{amssymb}

\newcommand{\eps}{\varepsilon}

\newcommand{\B}{{\mathcal B}}

\newcommand{\N}{{\mathbb N}}
\newcommand{\C}{{\mathbb C}}

\newcommand{\tef}{transcendental entire function}

\newcommand\sw{spider's web}

\theoremstyle{plain}
\newtheorem{theorem}{Theorem}[section]
\newtheorem{corollary}[theorem]{Corollary}

\newtheorem{lemma}[theorem]{Lemma}
\theoremstyle{definition}

\theoremstyle{remark}

\theoremstyle{problem}

\theoremstyle{example}

\begin{document}


\title[]{A sharp growth condition for a fast escaping spider's web}

\author{P. J. Rippon}
\address{Department of Mathematics and Statistics \\
The Open University \\
   Walton Hall\\
   Milton Keynes MK7 6AA\\
   UK}
\email{p.j.rippon@open.ac.uk}

\author{G. M. Stallard}
\address{Department of Mathematics and Statistics \\
The Open University \\
   Walton Hall\\
   Milton Keynes MK7 6AA\\
   UK}
\email{g.m.stallard@open.ac.uk}





\begin{abstract}
 We show that the fast escaping set $A(f)$ of a transcendental entire function $f$ has a structure known as a spider's web whenever the maximum modulus of $f$ grows below a certain rate. We give examples of entire functions for which the fast escaping set is not a spider's web which show that this growth rate is best possible. By our earlier results, these are the first examples for which the escaping set has a spider's web structure but the fast escaping set does not. These results give new insight into a conjecture of Baker and a conjecture of Eremenko.
\end{abstract}

\maketitle

\section{Introduction}
\setcounter{equation}{0}

Let $f:\C\to \C$ be a {\tef} and denote by
$f^{n},\,n=0,1,2,\ldots\,$, the $n$th iterate of~$f$. The {\it
Fatou set} $F(f)$ is the set of points $z \in \C$
such that $(f^{n})_{n \in \N}$ forms a normal
family in some neighborhood of $z$.  The complement of $F(f)$ is
called the {\it Julia set} $J(f)$ of $f$. An introduction to the
properties of these sets can be found in~\cite{wB93}.

In recent years, the escaping set defined by
\[
I(f) = \{z : f^n(z) \to \infty \mbox{ as } n \to \infty\}
\]
has come to play an increasingly significant role in the study of the iteration of {\tef}s with much of the research being motivated by a conjecture of Eremenko~\cite{E} that all the components of the escaping set are unbounded. For partial results on this conjecture see, for example,~\cite{lR07} and~\cite{RRRS}.

The most general result on Eremenko's conjecture was obtained in~\cite{RS05} where it was proved that the escaping set always has at least one unbounded component. This result was proved by considering the fast escaping set $A(f) = \bigcup_{n \in \N} f^{-n}(A_R(f))$, where
\[
A_R(f) = \{z:
            |f^n(z)| \geq M^n(R,f), \text{ for } n \in \N\}.
\]
Here
\[
 M(r) = M(r,f) = \max_{|z| = r} |f(z)|,
\]
$M^n(r,f)$ denotes the $n$th iterate of $M$ with respect to $r$, and $R>0$ is chosen so that $M(r,f) > r$ for $r\geq R$. The set $A(f)$ has many nice properties including the fact that all its components are unbounded -- these properties are described in detail in~\cite{RS09a}.

There are many classes of {\tef}s for which the fast escaping set has the structure of a spider's web -- see~\cite{RS09a},~\cite{MP11} and~\cite{Si}. We say that a set $E$ has this structure if $E$ is connected and there exists a sequence of bounded simply connected domains $G_n$ such that
\[
 \partial G_n \subset E, \; G_n \subset G_{n+1}, \mbox{ for } n \in \N, \mbox{ and } \bigcup_{n \in \N}G_n = \C.
\]
As shown in~\cite{RS09a}, if $A_R(f)$ has this structure then so do both $A(f)$ and $I(f)$, and hence Eremenko's conjecture is satisfied. Also, the domains $G_n$ can be chosen so that $\partial G_n \subset A_R(f) \cap J(f)$ and so $f$ has no unbounded Fatou components. This gives a surprising link between Eremenko's conjecture and a conjecture of Baker that all the components of the Fatou set are bounded if $f$ is a {\tef} of order less than $1/2$. Recall that the {\it order} of a {\tef} $f$ is defined to be
\[
 \rho = \limsup_{r \to \infty} \frac{\log \log M(r)}{\log r}.
\]

For background and recent results on Baker's conjecture, see~\cite{H},~\cite{HM},~\cite{RS08} and~\cite{RS11c}. It was shown in~\cite{RS08} (see also~\cite{RS09a}) that all earlier partial results on Baker's conjecture are in fact sufficient to imply the stronger result that $A_R(f)$ is a spider's web. Here we give a sharp condition on the growth of the maximum modulus that is sufficient to imply that $A_R(f)$ is a spider's web and hence that Baker's conjecture and Eremenko's conjecture are both satisfied. More precisely, we prove the following sufficient condition.

\begin{theorem}\label{cond}
Let $f$ be a transcendental entire function and let
$R>0$ be such that $M(r,f) > r$ for $r \geq R$. Let
\[
  R_n = M^n(R) \; \mbox{ and } \; \eps_n = \max_{R_n \leq r \leq R_{n+1}} \frac{\log \log M(r)}{\log
  r}.
\]
If
\[
  \sum_{n \in \N} \eps_n < \infty,
\]
then $A_R(f)$ is a spider's web.
\end{theorem}

We obtained a closely related result in~\cite[Theorem 3]{RS08} with the stronger hypothesis that $\sum_{n \in \N} \sqrt{\eps_n} < \infty$ and remarked there that the square root could be removed by introducing a more sophisticated argument. The method of proof given here is quite different, and more enlightening, than that used to prove~\cite[Theorem 3]{RS08}. In fact, Theorem~\ref{cond} follows surprisingly easily from a new local version of the classical $\cos \pi \rho$ theorem; see Theorem~\ref{cos}.

{\it Remark.} Theorem~\ref{cond}, can be generalised to apply to the set of points that escape as fast as possible within a direct tract of a transcendental meromorphic function; see~\cite{BRS} for earlier results concerning the fast escaping set in a direct tract.

It turns out that the condition in Theorem~\ref{cond} is, in a strong sense, best possible. In particular, the following result shows that the condition in Theorem~\ref{cond} cannot be replaced by the weaker condition that $\sum_{n \in \N} (\eps_n)^c < \infty$, for some $c>1$.

\begin{theorem}\label{ex}
There exist {\tef}s of the form
\begin{equation}\label{sym}
f(z)=  z^3 \prod_{n=1}^{\infty}\left(1+\frac{z}{a_n}\right)^{2p_n},
\end{equation}
where $p_n\in \N$, for $n\in \N$, and the sequence $(a_n)$ is positive and strictly increasing such that $A(f) \cap (-\infty,0] = \emptyset$; in particular, $A(f)$ is not a spider's web.

Moreover, if $(\delta_n)$ is a positive sequence such that
\[
  \sum_{n \in \N} \delta_n = \infty,
\]
then we can choose the sequence $(a_n)_{n \in \N}$ and a value $R>0$ in such a way that, with
\[
p_n = [a_n^{\delta_n/4}/4],\; \; R_n = M^n(R) \; \;  \mbox{ and } \; \;  \eps_n =
  \max_{R_n \leq r \leq R_{n+1}} \frac{\log \log M(r)}{\log r},
\]
there exists a subsequence $(n_k)$ such that
\begin{equation}\label{eps1}
 \eps_{n_k} \leq \delta_k + \frac{1}{2^{n_k}}, \;  \mbox{ for } k \in \N,
  \end{equation}
   and
 \begin{equation}\label{eps2}
  \eps_{n_k + m} \leq \frac{\delta_k}{3^{m-1}} + \frac{1}{2^{n_k+m}}, \;  \mbox{ for } k \in \N, 1 \leq m < n_{k+1} - n_k.
\end{equation}
\end{theorem}

Since it is possible to choose a positive sequence $(\delta_n)$ with
\[
\sum_{n \in \N} \delta_n = \infty \mbox{ and } \lim_{n \to \infty} \delta_n = 0,
 \]
 Theorem~\ref{ex} implies that there are functions of order zero for which $A_R(f)$ fails to be a spider's web. Thus new techniques are needed in order to solve Baker's conjecture. One such technique is introduced in~\cite{RS11c} where we show that all functions of order less than $1/2$ with zeros on the negative real axis satisfy Baker's conjecture and also satisfy Eremenko's conjecture with $I(f)$ being a spider's web. Since functions of the form~\eqref{sym} with $\limsup_{n \to \infty} \eps_n < 1/2$ are of this type, this gives the following corollary to Theorem~\ref{ex}, which answers a question in~\cite{RS09a}.

\begin{corollary}\label{SW}
There exist {\tef}s for which $I(f)$ is a spider's web but $A(f)$ is not a spider's web.
\end{corollary}

{\it Remark.} In fact we show in~\cite{RS11c} that functions of order less than $1/2$ with zeros on the negative real axis have the stronger property that $Q(f)$ contains a spider's web, where $Q(f)$ is the quite fast escaping set. Thus Theorem~\ref{ex} provides examples of functions for which $Q(f) \neq A(f)$; these two sets are equal for many functions, including all functions in the Eremenko-Lyubich class $\B$ as we show in~\cite{RS11a}.

The paper is arranged as follows. In Section 2 we prove Theorem~\ref{cond} and then, in Section~3, we prove Theorem~\ref{ex}.

\section{Proof of Theorem~\ref{cond}}
\setcounter{equation}{0}

Let $f$ be a {\tef} and $R>0$ be such that $M(r) > r$ for $r \geq R$. Recall that
\[
A_R(f) = \{z: |f^n(z)| \geq M^n(R), \mbox{ for } n \in \N\}
\]
and that $A_R(f)$ is a spider's web if $A_R(f)$ is connected and there exists
a sequence of bounded simply connected domains $G_n$ such that
\[
 \partial G_n \subset A_R(f), \; \; G_n \subset G_{n+1}, \mbox{ for } n \in \N, \; \; \mbox{ and } \; \; \bigcup_{n \in \N}G_n = \C.
\]
In this section we prove Theorem~\ref{cond} which gives a condition that is sufficient to ensure that $A_R(f)$ is a spider's web. The key ingredient in our proof is the following result which can be viewed as a local version of the classical cos $\pi \rho$ theorem. For a discussion of results of this type, see~\cite{RS12}.

\begin{theorem}\label{cos}
Let $f$ be a {\tef}. There exists $r(f)>0$ such that, if
\begin{equation}\label{growth}
 \log M(r) \leq r^{\alpha} \; \;  \mbox{ and } \; \; r^{1 - 2\alpha} \geq r(f),
\end{equation}
for some $\alpha \in (0,1/2)$,
then there exists $t \in (r^{1-2\alpha},r)$ such that
\[
\log m(t) >  \log M(r^{1-2\alpha}) - 2.
\]
\end{theorem}
\begin{proof}
We apply the following result of Beurling \cite[page~96]{aB33}:

Let $f$ be analytic in $\{z:|z| < r_0\}$, let $0\le r_1<r_2< r_0$, and put
\[
E=\{t\in (r_1,r_2):m(t)\le \mu\}, \;\;\text{where } 0<\mu<M(r_1).
\]
Then
\begin{equation}\label{estimate}
\log \frac{M(r_2)}{\mu}> \frac12\exp\left(\frac12\int_E \frac{dt}{t}\right)\log \frac{M(r_1)}{\mu}.
\end{equation}

Taking $r_2=r$, $r_1=r^{1-2\alpha}$, $\mu = M(r^{1-2\alpha})/e^2$, and $r(f)>0$ such that $M(r(f)) \geq e^2$, we deduce from~\eqref{growth} and~\eqref{estimate} that, if $m(t) \leq \mu$ for $t \in (r^{1-2\alpha},r)$, then
\[
r^{\alpha} \geq \log M(r) \geq \log \frac{M(r)}{\mu}> \frac12\exp\left(\frac12\int_{r^{1-2\alpha}}^r \frac{dt}{t}\right)\log \frac{M(r^{1-2\alpha})}{\mu} = r^\alpha.
\]
This is a contradiction and so there must exist $t \in (r^{1-2\alpha},r)$ such that $m(t) > \mu$; that is,
\[
\log m(t) > \log \mu = \log M(r^{1-2\alpha}) - 2,
\]
as required.\end{proof}

We also use the following results about spiders' webs proved in~\cite{RS09a}.


\begin{lemma}\label{m(r)}~\cite[Corollary 8.2]{RS09a}
Let $f$ be a {\tef} and let $R>0$ be such that $M(r) > r$ for $r\geq R$. Then $A_R(f)$ is a {\sw} if there exists a sequence $(\rho_n)$ such that, for $n \geq 0$,
\begin{equation}\label{r1}
\rho_n > M^n(R)
\end{equation}
and
\begin{equation}\label{r2}
m(\rho_n) \geq \rho_{n+1}.
\end{equation}
\end{lemma}

\begin{lemma}\label{SW1}~\cite[Lemma 7.1(d)]{RS09a}
Let $f$ be a {\tef}, let $R>0$ be such that $M(r) > r$ for $r\geq R$, and let $R'>R$. Then $A_R(f)$ is a spider's web if and only if $A_{R'}(f)$ is a spider's web.
\end{lemma}

In addition, we need the following property of the maximum modulus function, which was proved in this form in~\cite{RS08}.

\begin{lemma}\label{convex}~\cite[Lemma 2.2]{RS08}
Let $f$ be a {\tef}. Then there exists $R>0$ such that, for all $r \geq R$ and all $c>1$,
\[
  M(r^c) \geq M(r)^c.
\]
\end{lemma}

We are now in a position to prove Theorem~\ref{cond}.

\begin{proof}[Proof of Theorem~\ref{cond}]
Let $R>0$ be such that, for $r \geq R$, Lemma~\ref{convex} holds and $M(r) > r$. For $n \in \N$, let
\[
  R_n = M^n(R) \; \;  \mbox{ and } \; \;  \eps_n = \max_{R_n \leq r \leq R_{n+1}} \frac{\log \log M(r)}{\log
  r}.
\]
Suppose that $\sum_{n \in \N} \eps_n < \infty$. Then we can take $N$ sufficiently large to ensure that
\begin{equation}\label{size1}
\sum_{n \geq N} \eps_n < \frac{1}{8},
\end{equation}
and
\begin{equation}\label{size2}
  M(R_n)^{1/(8n^2)} = R_{n+1}^{1/(8n^2)} \geq e^2, \mbox{ for } n \geq N, \mbox{ and }
R_{N+1}^{1/4} \geq R_N \geq r(f),
\end{equation}
where $r(f)$ is as defined in Theorem~\ref{cos}. Note that~\eqref{size2} is possible since $\log M(r) / \log r \to \infty$ and so, for large $n$, we have $\log R_{n+1} > 4 \log R_n$.

Now let
\[
r_n = M^{n+1}\left( R_{N+1}^{\prod_{m=N}^{N+n}(1 - 2\eps_m - 1/(8m^2)} \right), \; \mbox{ for } \; n \geq 0.
 \]
 We note that, for $n \geq 0$, it follows from~\eqref{size1} that
\[
\prod_{m=N}^{N+n}\left(1 - 2\eps_m - \frac{1}{8m^2}\right) > 1 - \sum_{m=N}^{N+n}2 \eps_m - \sum_{m=N}^{N+n}\frac{1}{8m^2} \geq \frac{1}{2}
\]
and so, by~\eqref{size2},
\[
R_{N+n+2} > r_n > M^{n+1}(R_{N+1}^{1/2}) \geq M^{n+1}(R_N^2) = R_{N+n+1}^2.
\]
We claim that, for $n \geq 0$, there exists $\rho_n \in (R_{N+n+1},r_n)$ with $m(\rho_n) > r_{n+1}$. Indeed, it follows from Theorem~\ref{cos},~\eqref{size1},~\eqref{size2} and Lemma~\ref{convex} that, for $n \geq 0$, there exists $\rho_n \in (r_n^{1-2\eps_{n+N+1}},r_n) \subset (R_{N+n+1},r_n)$ such that
\begin{eqnarray*}
m(\rho_n) & \geq & \frac{1}{e^2}M(r_n^{1-2\eps_{n+N+1}})\\
 & \geq & M(r_n^{1-2\eps_{n+N+1}})^{1-1/(8(n+N+1)^2)}\\
 & \geq & M(r_n^{(1-2\eps_{n+N+1})(1-1/(8(n+N+1)^2))})\\
 & \geq & M(r_n^{(1-2\eps_{n+N+1} -1/(8(n+N+1)^2))})\\
 & = & M\left( \left( M^{n+1}\left( R_{N+1}^{\prod_{m=N}^{N+n}(1 - 2\eps_m - 1/(8m^2)} \right)\right )^{(1-2\eps_{n+N+1} -1/(8(n+N+1)^2))}\right)\\
 & \geq & M^{n+2} \left(  R_{N+1}^{\prod_{m=N}^{N+n+1}(1 - 2\eps_m - 1/(8m^2)} \right)\\
 & = & r_{n+1}.
\end{eqnarray*}
Thus, for $n \geq 0$, there exists $\rho_n > R_{N+n}$ with $m(\rho_n) \geq \rho_{n+1}$ and so, by Lemma~\ref{m(r)}, $A_{R_{N+1}}(f)$ is a {\sw}. It now follows from Lemma~\ref{SW1} that $A_R(f)$ is a {\sw} as claimed.
\end{proof}

\section{Proof of Theorem~\ref{ex}}
\setcounter{equation}{0}

Let
\[
f(z)=  z^3 \prod_{n=1}^{\infty}\left(1+\frac{z}{a_n}\right)^{2p_n},
\]
where the sequence $(a_n)$ is positive and strictly increasing.
In addition, let $(\delta_n)$ be a positive sequence such that
\[
  \sum_{n \in \N} \delta_n = \infty,
\]
and let
\begin{equation}\label{def}
p_n = [a_n^{\delta_n/4}/4].
\end{equation}
Without loss of generality, we assume that
\begin{equation}\label{del}
\delta_n < 1/2, \mbox{ for } n \in \N.
\end{equation}
Note that $f((-\infty,0]) \subset (-\infty,0]$ and that $m(r) = f(-r)$ and $M(r) = f(r) > r^3$, for $r > 0$. Further, $M(r) > r$ for $r\geq 1$.

We first show that the sequence $(a_n)$ can be chosen so that $A(f) \cap (-\infty,0] = \emptyset$.

We choose the values of $a_n$ carefully, beginning with $a_1$, then $a_2$ and so on.
Because of the way in which we choose the values of $a_n$, it is helpful
to introduce the function $g$ defined by
\begin{equation}\label{g}
 g(r)=\left\{
 \begin{array}{ll}
 r^3, &0 \leq r < a_1,\\
 r^3\displaystyle\prod_{a_n \leq r} \left( 1 + \frac{r}{a_n} \right)^{2p_n}, &r \geq a_1.
 \end{array}
 \right.
\end{equation}

Note that $g$ is a strictly increasing function and that it is discontinuous at $a_n$, for $n \in \N$. A key property of $g$ which we use repeatedly is that
\begin{equation}\label{fgM}
m(r) = -f(-r) < g(r) < M(r), \mbox{ for } r \geq 0.
\end{equation}
Since $g$ is increasing,~\eqref{fgM} implies that
\begin{equation}\label{fg}
f([-r,0]) \subset [-g(r),0], \mbox{ for } r \geq 0.
\end{equation}
We now set $r_0 = 10$ and $r_{n+1} = g(r_n) = g^{n+1}(10)$, for $n \in \N$, and note that
\begin{equation}\label{rn}
r_{n+1} \geq r_n^3, \mbox{ for } n \geq 0.
\end{equation}
Also, it follows from~\eqref{fg} that
\begin{equation}\label{rn1}
f^n((-r_m,0]) \subset (-r_{m+n},0], \mbox{ for } n,m \in \N.
\end{equation}

We begin by proving the following result.

 \begin{lemma}\label{Nk}
If there exists a sequence $(N_k)$ such that,
\begin{equation}\label{A1}
f^{N_1}((-r_2,0]) \subset (-r_{N_1},0]
\end{equation}
and, for $k \geq 2$,
\begin{equation}\label{Af}
f^{N_k}((-r_{N_1 + \cdots + N_{k-1} + 2k},0]) \subset (-r_{N_1 + \cdots + N_k},0],
\end{equation}
then $A(f) \cap (-\infty,0] = \emptyset$.
\end{lemma}
\begin{proof}
We first note that, if the hypotheses of Lemma~\ref{Nk} hold, then it follows from~\eqref{rn1} and~\eqref{Af} that, for $k \in \N$,
\begin{eqnarray*}
f^{N_1 + \cdots + N_k}((-r_{2k},0]) & = & f^{N_k}(f^{N_1 + \cdots + N_{k-1}})((-r_{2k},0])\\
& \subset & f^{N_k}((-r_{N_1 + \cdots + N_{k-1}+2k},0])\\
& \subset &(-r_{N_1 + \cdots + N_k},0].
\end{eqnarray*}
Thus
\begin{equation}\label{Af2}
f^{N_1 + \cdots + N_k}((-r_{2k},0]) \subset (-r_{N_1 + \cdots + N_k},0].
\end{equation}
Now let $z \in (-\infty,0]$. There exists $K \in \N$ such that, for $k \geq K$, we have $z \in (-r_k,0]$ and hence, by~\eqref{rn1}, we have $f^k(z) \in (-r_{2k},0]$. Thus, by~\eqref{Af2} and~\eqref{fgM}, for $k \geq K$,
\[
 |f^{N_1 + \cdots + N_k+k}(z)| < r_{N_1 + \cdots + N_k} < M^{N_1 + \cdots + N_k}(10)
\]
and hence
\[
 z \notin \{z: |f^{n+k}(z)| \geq M^n(10) \mbox{ for } n \in \N\}.
\]
 Thus $A(f) \cap (-\infty,0] = \emptyset$ as required.
 \end{proof}

We will show that we can choose the values of $a_n$ in such a way that the hypotheses of Lemma~\ref{Nk} hold. In order to do this, it is helpful to set certain restrictions on our choice of values. Firstly, we choose $a_1$ and $a_{n+1}/a_n$, $n \in \N$, sufficiently large to ensure that
\begin{equation}\label{r1}
  a_1^{\delta_1/4} \geq 4, \; \; a_{n+1} > a_n^2, \; \; a_{n+1}^{\delta_{n+1}/2} > 16 a_n^{\delta_n}
 \end{equation}
 and
 \begin{equation}\label{r2}
a_{n+1}^{\delta_{n+1}/16} >  a_{n}^{\delta_{n}} \log a_{n+1}.
 \end{equation}
 We note that~\eqref{r1} implies that
 \begin{equation}\label{2pn}
 p_1 \geq 1 \; \; \mbox{ and } \; \;  p_{n+1} \geq 2p_n^2, \mbox{ for } n \in \N.
 \end{equation}

 We also place certain restrictions on our choice of the values of $a_n$ in relation to the values of $r_n$:
 \begin{equation}\label{R3}
 \mbox{ if } a_k \in [r_n, r_{n+1}), \mbox{ then } a_m \notin [r_n, r_{n+4})
 \mbox{ for } k,m \in \N, \; m \neq k.
\end{equation}

We now show that, in order to prove that the hypotheses of Lemma~\ref{Nk} hold, it is sufficient to prove the following result.

\begin{lemma}\label{step}
Suppose that, for some $m \in \N$, we have defined the values of $a_n$ for which $a_n \leq r_{m}$ in such a way that they satisfy~\eqref{r1},~\eqref{r2} and~\eqref{R3}. Then we can choose $N \in \N$ and the values of $a_n$ for which $r_m < a_n \leq r_{m+N-1}$ in such a way that they satisfy~\eqref{r1},~\eqref{r2} and~\eqref{R3} and, no matter how the later values of $a_n$ are chosen,
\[
f^N((-r_{m+1},0]) \subset (-r_{m+N},0].
\]
\end{lemma}

Proving Lemma~\ref{step} is the key part of the proof that we can choose the sequence $(a_n)$ so as to ensure that $A(f) \cap (-\infty,0] = \emptyset$. Before proving Lemma~\ref{step}, we show that, if this result holds, then the hypotheses of Lemma~\ref{Nk} also hold.  First, by applying Lemma~\ref{step} when $m=1$ we see that there exists $N_{1,1} \in \N$ and a choice of $a_n$ for $r_1 < a_n \leq r_{N_{1,1}}$ such that
\begin{equation}\label{sa}
 f^{N_{1,1}}((-r_2,0]) \subset (-r_{N_{1,1}+1},0].
\end{equation}
We then apply Lemma~\ref{step} with $m = N_{1,1}$ and deduce that there exists $N_{1,2} \in \N$ and a choice of $a_n$ for $r_{N_{1,1}} < a_n \leq r_{N_{1,1} + N_{1,2} - 1}$ such that
\begin{equation}\label{sb}
f^{N_{1,2}}((-r_{N_{1,1}+1},0]) \subset (-r_{N_{1,1} + N_{1,2}},0].
\end{equation}
It follows from~\eqref{sa} and~\eqref{sb} that
\[
f^{N_{1,1} + N_{1,2}}(-r_2,0]  \subset   f^{N_{1,2}}(-r_{N_{1,1}+1},0]
 \subset (-r_{N_{1,1} + N_{1,2}},0].
\]
Putting $N_1 = N_{1,1} + N_{1,2}$, we deduce that we can choose the values of $a_n$ for which $r_1 < a_n \leq r_{N_1-1}$ in such a way that
\[
f^{N_1}((-r_2,0]) \subset (-r_{N_1},0].
\]
Thus~\eqref{A1} holds.

Now suppose that, for some $k \geq2$, we have defined $N_j$, for $1 \leq j \leq k-1$, and defined $a_n$, for $r_1 < a_n \leq r_{N_1 + \cdots + N_{k-1} - 1}$. We claim that we can use Lemma~\ref{step} to define $N_k \in \N$ and $a_n$ with $r_{N_1 + \cdots + N_{k-1} - 1} < a_n \leq r_{N_1 + \cdots + N_{k} - 1}$ such that~\eqref{Af} holds for $k$. The argument is similar to that given above. First, we apply Lemma~\ref{step} with $m = N_1 + \cdots + N_{k-1} + 2k-1$ to construct $N_{k,1}$ and $a_n$ with
\[
r_{N_1 + \cdots + N_{k-1} + 2k-1} < a_n \leq r_{N_1 + \cdots + N_{k-1} + N_{k,1} + 2k -2}
\]
such that
\[
f^{N_{k,1}}((-r_{N_1 + \cdots + N_{k-1} + 2k},0])
\subset (-r_{N_1 + \cdots + N_{k-1} + N_{k,1} + 2k -1},0].
\]

Then, for $2 \leq j \leq 2k$, we apply Lemma~\ref{step} repeatedly with 
\[
m = N_1 + \cdots + N_{k-1} + N_{k,1} + \cdots + N_{k,j-1} + 2k -j
\]
to construct $N_{k,j}$ and $a_n$ with
\[
r_{N_1 + \cdots + N_{k-1} + N_{k,1} + \cdots + N_{k,j-1} + 2k -j} < a_n \leq r_{N_1 + \cdots + N_{k-1} + N_{k,1} + \cdots + N_{k,j} + 2k -j-1}
\]
such that
\[
f^{N_{k,j}}((-r_{N_1 + \cdots + N_{k-1} + N_{k,1} + \cdots + N_{k,j-1} + 2k -j+1},0])
\subset (-r_{N_1 + \cdots + N_{k-1} + N_{k,1} + \cdots + N_{k,j} + 2k -j},0].
\]
Putting $N_k= N_{k,1} + \cdots + N_{k,2k}$, we deduce that $a_n$ can be chosen with
\[
r_{N_1 + \cdots + N_{k-1}-1} < a_n \leq r_{N_1 + \cdots + N_k - 1}
\]
such that
\[
f^{N_k}((-r_{N_1+ \cdots + N_{k-1} + 2k},0]) \subset (-r_{N_1 + \cdots + N_k},0]
\]
and hence~\eqref{Af} holds for $k$.

So, it remains to prove Lemma~\ref{step}.

We begin by proving four lemmas. The first describes the extent to which $f$ is small close to a zero at $-a_k$, where $k \in \N$.

\begin{lemma}\label{small}
 For each $k \in \N$,
\begin{equation}\label{one}
 |f(z)| < 1, \mbox{ for } z \in (-a_k,-a_k^{1-\delta_k/16}).
\end{equation}
\end{lemma}
\begin{proof}
This holds since, for such a $z$, it follows from~\eqref{def},~\eqref{r1},~\eqref{r2} and~\eqref{2pn} that
\begin{eqnarray*}
|f(z)| & \leq & a_k^3 \left(1-\frac{a_k^{1 - \delta_k/16}}{a_k}\right)^{2p_k} \prod_{m=1}^{k-1}\left(1+\frac{a_k}{a_m}\right)^{2p_m}
\prod_{m \geq k+1} \left(1+\frac{a_k}{a_m}\right)^{2p_m}\\
& \leq & \left(1-\frac{1}{a_k^{\delta_k/16}}\right)^{a_k^{\delta_k/4}} a_k^{3 + 2p_1 + \cdots + 2p_{k-1}}
\prod_{m \geq k+1} \left(1+\frac{1}{a_m^{1-1/2^{m-k}}}\right)^{a_m^{1/2}}\\
& \leq & \exp(-a_k^{\delta_k/16}) a_k^{a_{k-1}^{\delta_{k-1}/2}}e^{1 + 1/2 + 1/4 + \cdots}\\
& \leq &  a_k^{a_{k-1}^{\delta_{k-1}}} \exp(-a_k^{\delta_k/16}) < 1.
\end{eqnarray*}
\end{proof}

The second lemma shows that there is a large increase in the size of $g(r)$ at $r=a_k$, where $k \in \N$.

\begin{lemma}\label{large}
For each $k \in \N$,
\[
\log g(a_k) \geq p_k^{1/2} \log g(a_k^{1 - \delta_k/16}).
\]
\end{lemma}
\begin{proof}
For $k \in \N$, it follows from~\eqref{r1} that
\begin{eqnarray*}
g(a_k^{1 - \delta_k/16}) & < & a_k^3 \prod_{m=1}^{k-1}\left( 1 + \frac{a_k}{a_m}\right)^{2p_m}\\
& < & a_k^{3 + 2\sum_{m=1}^{k-1}p_m} \leq a_k^{4p_{k-1}}
\end{eqnarray*}
and
\[
g(a_k) \geq 2^{2p_k}.
\]
Thus, by~\eqref{r1},~\eqref{r2} and~\eqref{2pn},
\[
 \frac{\log g(a_k)}{\log g(a_k^{1 - \delta_k/16})} \geq \frac{2p_k \log 2}{4p_{k-1} \log a_k} > \frac{p_k}{3p_{k-1}\log a_k} > p_k^{1/2}.
\]
\end{proof}

The third lemma shows that $\log g$ has a convexity property.

\begin{lemma}\label{g1}
Let $r>0$ and $t \geq 2$. Then
\[
\log g(r^t) \geq t \log g(r).
\]
\end{lemma}
\begin{proof}
Let $r>0$ and $t\geq 2$. We have
\[
 g(r^t) \geq r^{3t}\prod_{a_m \leq r} \left( 1 + \frac{r^t}{a_m}\right)^{2p_m}
\]
and
\[
 g(r)^t = r^{3t} \prod_{a_m \leq r} \left( 1 + \frac{r}{a_m}\right)^{2p_mt}.
\]
Thus it is sufficient to show that
\[
 \left( 1 + \frac{r}{a_m}\right)^t \leq \left( 1 + \frac{r^t}{a_m}\right),
\]
when $a_m \leq r$. This is true since it follows from~\eqref{r1} that, for $a_m \leq r$ and $t \geq 2$,
\[
\left( 1 + \frac{r}{a_m}\right)^t \leq \left(\frac{r}{a_m^{1/2}}\right)^t = \frac{r^t}{a_m^{t/2}} <  1 + \frac{r^t}{a_m}.
\]
\end{proof}

The fourth lemma gives an upper bound on the growth of $g$ on intervals where no point is the modulus of a zero of $f$.

\begin{lemma}\label{g2}
Let $r>0$, $0 < s < 1/2$ and $t>1$ and suppose that there are no values of $n \in \N$ for which $a_n \in (r^s,r^t]$. Then
\[
\log g(r^t) \leq t(1+2s) \log g(r).
\]
\end{lemma}
\begin{proof}
It follows from~\eqref{r1} that
\[
g(r^t) = r^{3t}\prod_{a_m \leq r^s}\left( 1 + \frac{r^t}{a_m}\right)^{2p_m} <
 r^{3t}\prod_{a_m \leq r^s}r^{2p_mt} = r^{t(3 + \sum_{a_m \leq r^s}2p_m) }
\]
and
\[
g(r) > r^3\prod_{a_m \leq r^s}\left( \frac{r}{a_m}\right)^{2p_m} > r^{3 + \sum_{a_m \leq r^s}2p_m(1-s)}.
\]
Thus
\[
\log g(r^t)/ \log g(r) < t/(1-s) \leq t(1 + 2s),
\]
since $s < 1/2$.
\end{proof}

We are now in a position to prove Lemma~\ref{step}.

\begin{proof}[Proof of Lemma~\ref{step}]
Suppose that $m \in \N$ and that we have defined the values of $a_n$ for which $a_n \leq r_m$. We now define a sequence $(s_k)$, $0 \leq k \leq N$, inductively according to certain rules that we give below. Each time we define a value $s_k$, we also add a zero of $f$ at $-s_k$ provided this is allowed by~\eqref{r1},~\eqref{r2} and ~\eqref{R3}; no other zeros of $f$ are added. We choose our values $s_k$ in such a way that
\begin{equation}\label{sk}
r_{m+k} \leq s_k \leq r_{m+k+1}, \mbox{ for } 0 \leq k < N,
\end{equation}
\begin{equation}\label{sN}
s_N \leq r_{m+N}
\end{equation}
and
\begin{equation}\label{rs}
f^k((-r_{m+1},0]) \subset (-s_k,0], \mbox{ for } 0 \leq k \leq N.
\end{equation}
The result of Lemma~\ref{step} follows directly from~\eqref{sN} and~\eqref{rs}. The difficult part of the proof is to show that there exists an $N \in \N$ for which~\eqref{sN} is satisfied.

We define our sequence $(s_k)$ as follows:
\begin{itemize}
\item set $s_0 = r_{m+1}$;

\item if $s_k > r_{m+k}$ and there is a zero of $f$ at $-s_k$, then we set
 \begin{equation}\label{sz}
 s_{k+1} = g(s_k^{1-\delta_{n_k}/16});
 \end{equation}

\item if $s_k > r_{m+k}$ and there is no zero of $f$ at $s_k$, then we set
\begin{equation}\label{snotz}
s_{k+1} = g(s_k);
\end{equation}

\item if $s_k \leq r_{m+k}$, then we terminate the sequence $(s_k)$.
\end{itemize}

It follows from Lemma~\ref{small} and~\eqref{fg} that, with this construction,~\eqref{sk},~\eqref{sN} and~\eqref{rs} are indeed satisfied.

It remains to prove that there exists $K \in \N$ such that the sequence terminates at $s_K$; that is, if
\[
T_k = \frac{\log s_k}{\log r_{m+k}},
\]
then there exists $K \in \N$ such that $T_K \leq 1$.

We introduce the following terminology. We let $L$ denote the largest integer for which $a_L \leq r_m$ and define a (finite) subsequence $(k_n)$ such that
\begin{equation}\label{skn}
a_{L+n} = s_{k_n}, \mbox{ for } n = 1,2,\ldots.
\end{equation}

The main idea is to show that, for each $n\geq 2$ we have that $T_{k_n + 1}$ is less than $T_{k_n}$, with $k_n$ defined as above. These decreases counteract the small increases that may occur from $T_k$ to $T_{k+1}$ for other values of $k$ and, for $n$ large enough, they will combine together to cause $T_{k_n +1}$ to drop below $1$.

We first estimate some quantities that will be useful in our calculations. We begin by noting that it follows from~\eqref{skn},~\eqref{sk} and Lemma~\ref{large} that, for $n \geq 1$,
\begin{eqnarray*}
\log r_{m+k_n+2} & = & \log g(r_{m+k_n + 1}) \\
& \geq & \log g(s_{k_n}) \geq p_{L+n}^{1/2} \log g(s_{k_n}^{1 - \delta_{L+n}/16}).
\end{eqnarray*}
Thus, by~\eqref{sz}
\begin{equation}\label{rps}
\log r_{m+k_n+2} \geq p_{L+n}^{1/2} \log s_{k_n + 1}, \mbox{ for } n \geq 1.
\end{equation}
Together with~\eqref{rn},~\eqref{rps} implies that

\begin{equation}\label{e1}
\log r_{m+k_n+q} \geq 3^{q-2}p_{L+n}^{1/2} \log s_{k_n + 1}, \mbox{ for }  q \geq 2, \; n \geq 1.
\end{equation}

Together with Lemma~\ref{g1},~\eqref{rps} implies that
\begin{equation}\label{f}
\frac{\log s_{k_n + q}}{\log r_{m + k_n + q + 1}} \leq \frac{\log g^{q-1}(s_{k_n + 1})}{\log g^{q-1}(r_{m + k_n + 2})}
\leq \frac{\log s_{k_n + 1}}{\log r_{m + k_n + 2}} \leq \frac{1}{p_{L+n}^{1/2}}, \mbox{ for } q \geq 2, \; n \geq 1.
\end{equation}

Now fix $n\geq 2$ and write
\[
t_{n,q} = T_{k_n + q} =  \frac{\log s_{k_n+q}}{\log r_{m + k_n +q}}, \mbox{ for }  q \geq 2.
\]
For $2 \leq q < k_{n+1} - k_n$, there are no zeros of $f$ with modulus in the interval $(s_{k_n},s_{k_{n+q}})$ and so it follows from~\eqref{snotz},~Lemma~\ref{g2} and~\eqref{e1} that, for such $q$,
\begin{eqnarray*}
\log s_{k_n + q + 1} & = & \log g(s_{k_n + q})\\
& \leq & t_{n,q}\left( 1 + 2 \frac{\log s_{k_n}}{\log r_{m + k_n + q}} \right) \log g(r_{m + k_n + q})\\
& = & t_{n,q} \left( 1 + 2 \frac{\log s_{k_n}}{\log r_{m + k_n + q}} \right) \log r_{m + k_n + q + 1}\\
& \leq & t_{n,q} \left( 1 +  \frac{2}{3^{q-2}p_{L+n}^{1/2}} \right) \log r_{m + k_n + q + 1}.
\end{eqnarray*}
Thus, for $2 \leq q < k_{n+1} - k_n$, we have
\begin{equation}\label{e2}
 t_{n,q+1} \leq t_{n,q}  \left( 1 + \frac{2}{3^{q-2}p_{L+n}^{1/2}}\right).
\end{equation}

For $q = k_{n+1} - k_n$, there are no zeros of $f$ with modulus in the interval $(s_{k_n},s_{k_{n+q}})$ and so it follows from~\eqref{sz},~Lemma~\ref{g2} and~\eqref{e1} that
\begin{eqnarray*}
\log s_{k_n + q + 1} & = & \log g(s_{k_n + q}^{1 - \delta_{L+n+1}/16})\\
& \leq & t_{n,q}\left( 1 - \frac{\delta_{L+n+1}}{16} \right) \left( 1 + 2 \frac{\log s_{k_n}}{\log r_{m + k_n + q}} \right) \log g(r_{m + k_n + q})\\
& = & t_{n,q} \left( 1 - \frac{\delta_{L+n+1}}{16} \right) \left( 1 + 2 \frac{\log s_{k_n}}{\log r_{m + k_n + q}} \right) \log r_{m + k_n + q + 1}\\
& \leq & t_{n,q} \left( 1 - \frac{\delta_{L+n+1}}{16} \right) \left( 1 +  \frac{2}{3^{q-2}p_{L+n}^{1/2}} \right) \log r_{m + k_n + q + 1}.
\end{eqnarray*}
Thus, for $q = k_{n+1} - k_n$, we have
\begin{equation}\label{e3}
t_{n,q+1} \leq t_{n,q} \left( 1 - \frac{\delta_{L+n+1}}{16} \right) \left( 1 +  \frac{2}{3^{q-2}p_{L+n}^{1/2}} \right).
\end{equation}

Lastly, it follows from~\eqref{R3} that, if $q = k_{n+1} - k_n + 1$, then $q-1 \geq 2$. Also, there are no zeros of $f$ with modulus in the interval $(s_{k_{n+1}},s_{k_{n+1}+1}) = (s_{k_{n+1}},s_{k_n+q})$ and so it follows from Lemma~\ref{g2} and~\eqref{f} that
\begin{eqnarray*}
\log s_{k_n + q + 1} & = & \log g(s_{k_n + q})\\
& \leq & t_{n,q}\left( 1 + 2 \frac{\log s_{k_{n+1}}}{\log r_{m + k_n + q}} \right) \log g(r_{m + k_n + q})\\
& = & t_{n,q}\left( 1 + 2 \frac{\log s_{k_{n} + q-1}}{\log r_{m + k_{n} + q}} \right) \log r_{m + k_n + q + 1}\\
& \leq &  t_{n,q}\left( 1 +  \frac{2}{p_{L+n}^{1/2}} \right) \log r_{m + k_n + q + 1}.\\
\end{eqnarray*}
Thus, for $q = k_{n+1} - k_n + 1$, we have
\begin{equation}\label{e4}
t_{n,q+1} \leq t_{n,q} \left( 1 +  \frac{2}{p_{L+n}^{1/2}} \right).
\end{equation}

It follows from~\eqref{e2},~\eqref{e3},~\eqref{e4} and~\eqref{2pn} that, for $M \geq 2$, we have
\begin{eqnarray*}
T_{k_{M+1} + 2} & = & t_{M, k_{M+1} - k_M + 2}\\
 & = & t_{2,2} \prod_{n=2}^M \prod_{q=2}^{k_{n+1} - k_n + 1} \frac{t_{n,q+1}}{t_{n,q}}\\
  & \leq & t_{2,2} \prod_{n = 2}^M \left( 1 + \frac{2}{p_{L+n}^{1/2}} \right)\left( 1 - \frac{\delta_{L+n+1}}{16}\right)
  \prod_{q=2}^{k_{n+1} - k_n}\left( 1 + \frac{2}{3^{q-2}p_{L+n}^{1/2}}\right)\\
  & \leq & t_{2,2} \prod_{n = 2}^M \left( \left( 1 + \frac{2}{p_{L+n}^{1/2}} \right)^3\left( 1 - \frac{\delta_{L+n+1}}{16}\right) \right) .
\end{eqnarray*}
 It follows from~\eqref{2pn} that $\sum_{n \in \N} \frac{1}{p_{L+n}^{1/2}} < \infty$ and so, since
$\sum_{n \in \N} \delta_{L+n + 1} = \infty$, we deduce that, for $M$ sufficiently large, $T_{k_{M+1} + 2} \leq 1$, as required.
\end{proof}

We have now proved Lemma~\ref{step}. As noted earlier, this is sufficient to imply that the hypotheses of Lemma~\ref{Nk} hold and hence that $A(f) \cap (-\infty,0] = \emptyset$ as required.

We complete the proof of Theorem~\ref{ex} by showing that, in addition, conditions~\eqref{eps1} and~\eqref{eps2} are satisfied. That is, we prove the following.

\begin{lemma}\label{conds}
Let
\begin{equation}\label{max}
  \eps_n =
  \max_{R_n \leq r \leq R_{n+1}} \frac{\log \log M(r)}{\log r}.
\end{equation}
There exists a subsequence $(n_k)$ such that
\begin{equation}\label{epsd1}
 \eps_{n_k} \leq \delta_k + \frac{1}{2^{n_k}}, \mbox{ for } k \in \N,
  \end{equation}
   and
   \begin{equation}\label{epsd2}
    \eps_{n_k + m} \leq \frac{\delta_k}{3^{m-1}} + \frac{1}{2^{n_k + m} }, \; \mbox{ for } k \in \N, 1 \leq m < n_{k+1} - n_k.
\end{equation}
\end{lemma}

\begin{proof}
 We begin by setting $R_0 = r_0 = 10$ and defining $R_{n+1} = M(R_n)$, for $n \in \N$. Clearly $R_n \geq r_n$ by~\eqref{fgM} and
 \begin{equation}\label{Rn}
 R_{n+1} \geq R_n^3, \mbox{ for } n \in \N.
 \end{equation}

 We claim that

\begin{equation}\label{r3}
 \mbox{ if } a_k \in [R_n, R_{n+1}), \mbox{ then } a_m \notin [R_n, R_{n+2})
 \mbox{ for } k,m \in \N, \; m \neq k.
\end{equation}
In order to deduce this from~\eqref{R3}, it is sufficient to show that, if $r_p \in [R_n, R_{n+1})$, for some $p,n \in \N$, then $r_{p+2} > R_{n+1}$. We prove this in two steps. Firstly, we note that if $r_p \in [R_n, R_n^3)$, for some $p,n \in \N$, then it follows from~\eqref{rn} that $r_{p+1} \geq r_p^3 \geq R_n^3$. Secondly, if $r_p \in [R_n^3,R_{n+1})$, for some $p,n \in \N$, then we claim that
\begin{equation}\label{second}
r_{p+1} = g(r_p) \geq g(R_n^3) > M(R_n) = R_{n+1}.
\end{equation}
This is true since, if $k$ is the smallest integer such that $a_k > R_n^3$, then
\[
g(R_n^3) = R_n^9 \prod_{m=1}^{k-1}\left(1 + \frac{R_n^3}{a_m} \right)^{2p_m}
\]
and so, by~\eqref{def} and~\eqref{r1},
\begin{eqnarray*}
M(R_n) = f(R_n) & = & R_n^3 \prod_{m=1}^{\infty}\left(1 + \frac{R_n}{a_m} \right)^{2p_m}\\
& < & \frac{g(R_n^3)}{R_n^6}\left(1 + \frac{R_n}{a_k} \right)^{2p_k}\prod_{m \geq k+1}\left(1 + \frac{a_k}{a_m} \right)^{2p_m}\\
& < & \frac{g(R_n^3)}{R_n^6} \left(1 + \frac{1}{a_k^{1/2}} \right)^{a_k^{1/2}} \prod_{m \geq k+1}\left(1 + \frac{1}{a_m^{1-1/2^{m-k}}} \right)^{a_m^{1/2}}\\
& \leq & \frac{g(R_n^3)}{R_n^6}e^{1+1+1/2+1/4+\cdots} < g(R_n^3).
\end{eqnarray*}
Thus~\eqref{second} does indeed hold and, by the reasoning above, this is sufficient to show that~\eqref{r3} holds.

Now, for $k \in \N$, we choose $n_k \in \N$ such that $a_k \in [R_{n_k}, R_{n_k + 1})$. Then, by~\eqref{r3}, this defines a sequence $(n_k)$ with $n_j \neq n_k$ for $j \neq k$. Now suppose that $r \in [R_{n_k},R_{n_k+1}]$, for some $k \in \N$. It follows from~\eqref{r1} and~\eqref{r3} that
\begin{eqnarray*}
M(r) = f(r) & \leq & r^3 \left(1+\frac{r}{a_k}\right)^{2p_k} \prod_{m=1}^{k-1}\left(1+\frac{r}{a_m}\right)^{2p_m}
\prod_{m \geq k+1} \left(1+\frac{r}{a_m}\right)^{2p_m}\\
& \leq & \left(1+\frac{r}{a_k}\right)^{2p_k} r^{3 + 2p_1 + \cdots + 2p_{k-1}}
\prod_{m \geq k+1} \left(1+\frac{1}{a_m^{1-1/2^{m-k}}}\right)^{a_m^{1/2}}\\
& < & \left(1+\frac{r}{a_k}\right)^{a_k^{\delta_k}} r^{a_{k-1}^{\delta_{k-1}}}e^{1 + 1/2 + 1/4 + \cdots}
\end{eqnarray*}
and so
\begin{equation}\label{M1}
M(r) <  e^2 r^{a_{k-1}^{\delta_{k-1}}}\left(1+\frac{r}{a_k}\right)^{a_k^{\delta_k}}.
\end{equation}

If $r < a_k^{1/2}$, then it follows from~\eqref{del} and~\eqref{M1} that
\[
M(r) < e^3 r^{a_{k-1}^{\delta_{k-1}}} < e^3r^{r^{\delta_k}}
\]
and hence, since $r \geq R_1 \geq 1000$,
\[
 \frac{\log \log M(r) }{\log r} < \frac{\delta_k \log r + 2\log \log r}{\log r} = \delta_k + 2 \frac{\log \log r}{\log r}
 \leq \delta_k + 2 \frac{\log \log R_{n_k}}{\log R_{n_k}}.
\]
It follows from~\eqref{Rn} that, in this case,
\begin{equation}\label{case1}
 \frac{\log \log M(r) }{\log r} \leq \delta_k + 2\frac{\log(3^{n_k} \log 10)}{3^{n_k} \log 10} < \delta_k + \frac{1}{2^{n_k}}.
\end{equation}

If $a_k^{1/2} \leq r \leq a_k$, then
\[
 \left(1+\frac{r}{a_k}\right)^{a_k^{\delta_k}} = \left(1+\frac{r}{a_k}\right)^{(a_k/r)^{\delta_k} r^{\delta_k}} < \left(1+\frac{r}{a_k}\right)^{(a_k/r)r^{\delta_k}} \leq e^{r^{\delta_k}}
\]
and, if $r > a_k$, then
\[
\left(1+\frac{r}{a_k}\right)^{a_k^{\delta_k}} < r^{a_k^{\delta_k}} < r^{r^{\delta_k}}.
\]
So, if $r \geq a_k^{1/2}$, it follows from~\eqref{M1} and~\eqref{r1} that
\[
M(r) < e^2 r^{a_{k-1}^{\delta_{k-1}}} r^{r^{\delta_k}} < e^2 r^{a_{k}^{\delta_k/2}}r^{r^{\delta_k}}  < e^2 r^{2r^{\delta_k}}
\]
and hence
\[
 \frac{\log \log M(r) }{\log r} < \frac{\delta_k \log r + 2\log \log r}{\log r} = \delta_k + 2 \frac{\log \log r}{\log r} \leq \delta_k + 2 \frac{\log \log R_{n_k}}{\log R_{n_k}}.
\]
As before, it follows from~\eqref{Rn} that
\begin{equation}\label{case2}
 \frac{\log \log M(r) }{\log r} \leq \delta_k + \frac{1}{2^{n_k}}.
\end{equation}
Together with~\eqref{case1}, this implies that~\eqref{epsd1} holds.

Now suppose that $r \in [R_{n_k + m}, R_{n_k + m +1})$, for some $k \in \N$, $1 \leq m < n_{k+1} - n_k$. It follows from~\eqref{r1} and~\eqref{Rn} that

\begin{eqnarray*}
M(r) = f(r) & \leq & r^3 \prod_{m=1}^{k}\left(1+\frac{r}{a_m}\right)^{2p_m}
\prod_{m \geq k+1} \left(1+\frac{r}{a_m}\right)^{2p_m}\\
& \leq &  r^{3 + 2p_1 + \cdots + 2p_{k}}
\prod_{m \geq k+1} \left(1+\frac{1}{a_m^{1-1/2^{m-k}}}\right)^{a_m^{1/2}}\\
& \leq & r^{a_{k}^{\delta_{k}}}e^{1 + 1/2 + 1/4 + \cdots}\\
& \leq & e^2 r^{a_{k}^{\delta_{k}}} \leq e^2r^{R_{n_k + 1}^{\delta_k}}\\
& < & e^2 r^{r^{\delta_{k}/3^{m-1}}}\\
\end{eqnarray*}
Thus
\[
 \frac{\log \log M(r) }{\log r} < \frac{\delta_k \log r/3^{m-1} + 2\log \log r}{\log r} < \frac{\delta_k}{3^{m-1}} + 2 \frac{\log \log r}{\log r} \leq \frac{\delta_k}{3^{m-1}} + 2 \frac{\log \log R_{n_k + m}}{\log R_{n_k+m}}.
\]
As before, it follows from~\eqref{Rn} that
\begin{equation}\label{case3}
 \frac{\log \log M(r) }{\log r} \leq \frac{\delta_k}{3^{m-1}} + \frac{1}{2^{n_m + m}}
\end{equation}
and so~\eqref{epsd2} holds.
\end{proof}

\end{document}